\documentclass[11pt, twoside]{article}
\usepackage{amsfonts}
\usepackage{}
\usepackage{amsmath,amsthm}
\usepackage{amssymb,latexsym}
\usepackage{mathrsfs}
\usepackage{enumerate}
\usepackage[colorlinks,citecolor=blue,dvipdfm]{hyperref}
\usepackage[numbers,sort&compress]{natbib}
\usepackage{indentfirst}

\newtheorem{theorem}{Theorem}[section]
\newtheorem{corollary}[theorem]{Corollary}

\newtheorem{remark}[theorem]{Remark}

\theoremstyle{definition} \theoremstyle{remark}
\numberwithin{equation}{section}
\allowdisplaybreaks

\date{}

\begin{document}

\markboth{\\ G. He et al.}{On the quasi-ergodic distribution of absorbing Markov processes}

%\date{November 24, 2016}

\title{{\bf On the quasi-ergodic distribution of absorbing Markov processes}}
\setcounter{footnote}{-1}
\author{Guoman He$^{1, 2}$\thanks{\small\footnotesize{$^{1}$School of Mathematics and Statistics, Hunan University of Commerce, Changsha, Hunan 410205, PR China.
~~Email address: hgm0164@163.com
\newline
\indent~~$^{2}$Key Laboratory of Hunan Province for New Retail Virtual Reality Technology, Hunan University of Commerce, Changsha, Hunan 410205, PR China}},~~~~~~
\setcounter{footnote}{-1}
Hanjun Zhang$^{3}$\thanks{\footnotesize {$^{3}$School of Mathematics and Computational Science, Xiangtan University, Xiangtan, Hunan 411105, PR China.
~Email address: hjz001@xtu.edu.cn}},~~~~~~
\setcounter{footnote}{-1}
Yixia Zhu$^{4}$\thanks{\footnotesize {$^{4}$School of Mathematics and Statistics, Hunan University of Finance and Economics, Changsha, Hunan 410205, PR China.
~Email address: zhuyixia62@163.com}}
\\[1.8mm]}

\baselineskip 0.23in

\maketitle

\begin{abstract}
In this paper, we give a sufficient condition for the existence of a quasi-ergodic distribution for absorbing Markov processes. Using an orthogonal-polynomial approach, we prove that the previous main result is valid for the birth--death process on the nonnegative integers with 0 an absorbing boundary and $\infty$ an entrance boundary. We also show that the quasi-ergodic distribution is stochastically larger than the unique quasi-stationary distribution in the sense of monotone likelihood-ratio ordering for the birth--death process.
\\[2mm]
{\bf Keywords:} Process with absorption; quasi-ergodicity; quasi-stationary distribution; birth--death process
\\[1.2mm]
{\bf 2010 MSC:} Primary 60J25; Secondary 37A30, 60J80
\end{abstract}

\baselineskip 0.234in
\section{Introduction}
\label{sect1}
Let $(\Omega, (\mathcal{F}_{t})_{t\geq0}, (X_{t})_{t\geq0}, (P_{t})_{t\geq0}, (\mathbb{P}_{x})_{x\in E\cup\{\partial\}})$ be a time-homogeneous Markov process with state space $E\cup\{\partial\}$, where $(E, \mathcal {E})$ is a measurable space and $\partial\not\in E$ is a cemetery state. Let $\mathbb{P}_{x}$ and $\mathbb{E}_{x}$ stand for the probability and the expectation, respectively, associated with the process $X$ when initiated from $x$. We assume that the process $X$ has a finite lifetime $T$, i.e.,
for all $x\in E$,
$$\mathbb{P}_{x}(T<\infty)=1, $$
where $T=\inf\{t\geq0: X_t=\partial\}$. We also assume that for all $t\geq0$ and $x\in E$,
$$\mathbb{P}_{x}(t<T)>0.$$
The main purpose of this work is to study the existence of a quasi-ergodic distribution for an absorbing Markov process, and give a comparison between the quasi-ergodic distribution and the quasi-stationary distribution for a class of birth--death process.
\par
A probability measure $\nu$ on $E$ is called a {\em quasi-stationary distribution} if, for all $t\geq0$ and any $A\in\mathcal {E}$,
\begin{equation*}
\mathbb{P}_{\nu}(X_t\in A|T>t)=\nu(A).
\end{equation*}
Quasi-stationary distribution for a killed Markov process has been studied by various authors since 1940s. On this topic, we refer the reader to survey papers
\cite{MV12, VDP13} and the book \cite{CMS13} for the background and more informations.
\par
A probability measure $m$ on $E$ is called a {\em quasi-ergodic distribution} if, for any $x\in E$ and any bounded measurable function $f$ on $E$, the following limit exists:
\begin{equation*}
\lim_{t\rightarrow\infty}\mathbb{E}_{x}\left(\frac{1}{t}\int_{0}^{t}f(X_s)ds|T>t\right)=\int_{E}f(x)m(dx).
\end{equation*}
\par
We remark that the above limiting law of the time-average, which we called quasi-ergodic distribution, comes from the paper \cite{BR99}, where the authors proved initially a conditioned version of the ergodic theorem for Markov processes. Under some mild conditions, Chen and Deng \cite{CD13} showed that a quasi-ergodic distribution can be characterized by the Donsker--Varadhan rate functional which is typically used as the large deviation rate function for Markov processes. Recently, many authors have extensively studied the quasi-ergodic distribution; see \cite{CLJ12,HZ16,ZLS14} for example. In existing research works, it often needs to assume the process is $\lambda$-positive (see, e.g., \cite{BR99, CLJ12, F74}), except some specific cases. However, for general, almost surely absorbed Markov processes, checking whether it is $\lambda$-positive is not an easy thing to do. This leads us to look for some alternative conditions ensuring the existence of quasi-ergodic distributions of absorbing Markov processes.
\par
Quasi-ergodic distribution, sometimes referred to as the limiting conditional mean ratio quasi-stationary distribution \cite{F74}, is quite different from quasi-stationary distribution (see, e.g., \cite{CLJ12,HZ16,ZLS14}). A natural question is whether there is a relationship between them$?$ In this paper, we plan to give a comparison between them for a class of birth--death process, who admits a quasi-ergodic distribution and a unique quasi-stationary distribution. We show that the quasi-ergodic distribution is stochastically larger than the unique quasi-stationary distribution in the sense of monotone likelihood-ratio ordering (see, e.g., \cite{W85}) for the process. In other words, the quasi-ergodic distribution provides an upper bound for the quasi-stationary distribution in the sense of monotone likelihood-ratio ordering. Although the quasi-stationary distributions of absorbing Markov processes are known to have considerable practical importance in, e.g., ecology, biological, and physical chemistry, computation of the quasi-stationary distributions is often nontrivial. Thus, the comparison result seems to be meaningful.
\par
In this work, we first prove that, under suitable assumptions, there exists a quasi-ergodic distribution for the absorbing Markov process. In order to illustrate the previous main result, we then consider a birth--death process on the nonnegative integers with $0$ as an absorbing state and $\infty$ as an entrance boundary. Based on orthogonal polynomial techniques of \cite{KM57b}, we show that the previous main result is valid for the birth--death process.
\par
The remainder of this paper is organized as follows. The main result and its proof are presented in Section \ref{sect2}. In Section \ref{sect3} we study the case of birth--death processes. We conclude in Section \ref{sect5} with an example.
\par

\section{Main result}
\label{sect2}
In this paper, our main goal is to prove that Assumption (A) below is a sufficient criterion for the existence of a quasi-ergodic distribution for an absorbing Markov process.
We point out that, for an absorbing Markov process, Assumption (A) is a necessary and sufficient condition for exponential convergence to a unique quasi-stationary distribution in the total variation norm (see \cite[Theorem 2.1]{CV16}).
\vskip0.2cm
\noindent{\bf Assumption (A)} There exists a probability measure $\nu_{1}$ on $E$ such that
\vskip0.1cm
\noindent(A1) there exist $t_{0}, c_{1}>0$ such that, for all $x\in E$,
\begin{equation*}
\mathbb{P}_{x}(X_{t_{0}}\in\cdot|T>t_{0})\geq c_{1}\nu_{1}(\cdot);
\end{equation*}
\noindent(A2) there exists $c_{2}>0$ such that, for all $x\in E$ and $t\geq0$,
\begin{equation*}
\mathbb{P}_{\nu_{1}}(T>t)\geq c_{2}\mathbb{P}_{x}(T>t).
\end{equation*}
\par
In order to let the reader have a better understanding of Assumption (A), its specific meaning (see \cite{CV16}) is restated here: If $E$ is a Polish space, then Assumption (A1) implies that the process $X$ comes back fast in compact sets from any initial conditions. If $E=\mathbb{N}$ or $\mathbb{R}_{+}:=(0, +\infty)$ and $\partial=0$, then Assumption (A1) implies that the process $X$ {\em comes down from infinity} (see \cite{CMS13}); Assumption (A2) means that the highest non-absorption probability among all initial points in $E$ has the same order of magnitude as the non-absorption probability starting from the probability distribution $\nu_{1}$.
\par
According to \cite[Proposition 2.3]{CV16}, we know that Assumption (A) implies that there exists a non-negative function $\eta$ on $E\cup\{\partial\}$, which is positive on $E$ and vanishes on $\partial$, such that
\begin{equation}
\label{2.1}
 \eta(x)=\lim_{t\to\infty}e^{\lambda t}\mathbb{P}_{x}(T>t),
\end{equation}
where the convergence holds for the uniform norm on $E\cup\{\partial\}$ and $\lambda>0$. Moreover, $\eta$ belongs to the domain of the infinitesimal generator $L$ of
the semigroup $(P_{t})_{t\geq0}$ on the set of bounded Borel functions on $E\cup\{\partial\}$ equipped with uniform norm and
\begin{equation*}
L\eta=-\lambda\eta.
\end{equation*}
\par
According to \cite[Theorem 3.1]{CV16}, we know that Assumption (A) implies that the $Q$-process $($the process conditioned to never be absorbed$)$ exists. More precisely, if Assumption (A) holds, then for all $A\in\mathcal{F}_{s}$ and all $s\geq0$, the family $(\mathbb{Q}_{x})_{x\in E}$ of probability measures on $\Omega$ defined by
\begin{equation*}
\mathbb{Q}_{x}(A)=\lim_{t\to\infty}\mathbb{P}_{x}(A|T>t)
\end{equation*}
is well-defined, and the process $(\Omega, (\mathcal{F}_{t})_{t\geq0}, (X_{t})_{t\geq0}, (\mathbb{Q}_{x})_{x\in E})$ is an $E$-valued homogeneous Markov process.
\par
The following theorem is our main result.
\begin{theorem}
\label{thm2.1}
Assume that Assumption $(\mathrm{A})$ holds. Then, there exists a quasi-ergodic distribution
\begin{equation*}
m(dx)=\eta(x)\nu(dx)
\end{equation*}
for the process $X$, where $\nu$ is the unique quasi-stationary distribution of the process $X$. In particular, $m$ is the unique stationary distribution of the $Q$-process.
\end{theorem}
\begin{proof}
From \cite[Proposition 2.3]{CV16}, we know that $\int_{E}\eta(x)\nu(dx)=1$. Then, $m$ is a probability distribution on $E$.
Next, we first assume that $f$ is positive and bounded. For fixed $u$, we set
\begin{equation*}
h_{u}(x)=\inf\{e^{\lambda r}\mathbb{P}_{x}(T>r)/\eta(x):r\geq u\}.
\end{equation*}
From (\ref{2.1}), one can easily see that $h_{u}(x)\uparrow1$, as $u\rightarrow\infty$. Let $0<q<1$. When $(1-q)t\geq u$, by the Markov property, we obtain
\begin{eqnarray*}
\mathbb{E}_{x}(f(X_{qt})|T>t)&=&\frac{\mathbb{E}_{x}(f(X_{qt}),T>t)}{\mathbb{P}_{x}(T>t)}\\
                             &=&\frac{\mathbb{E}_{x}[f(X_{qt}){\bf1}_{\{T>qt\}}\mathbb{P}_{X_{qt}}(T>(1-q)t)]}{\mathbb{P}_{x}(T>t)}\\
                             &=&\frac{e^{\lambda qt}\mathbb{E}_{x}[f(X_{qt}){\bf1}_{\{T>qt\}}e^{\lambda(1-q)t}\mathbb{P}_{X_{qt}}(T>(1-q)t)]}{e^{\lambda t}\mathbb{P}_{x}(T>t)}\\
                             &\geq&\frac{e^{\lambda qt}\mathbb{E}_{x}[f(X_{qt})h_{u}(X_{qt})\eta(X_{qt}){\bf1}_{\{T>qt\}}]}{e^{\lambda t}\mathbb{P}_{x}(T>t)},
\end{eqnarray*}
where ${\bf1}_{A}$ denotes the indicator function of $A$.
\par
From \cite[Proposition 2.3]{CV16}, we know that $\eta$ is bounded. Moreover, because the convergence in (\ref{2.1}) is uniform in $x\in E$, there
exists a constant $C>0$ such that, for all $r\geq u$ and $x\in E$,
\begin{equation*}
|f(x)h_{u}(x)\eta(x)|\leq|f(x)e^{\lambda r}\mathbb{P}_{x}(T>r)|\leq C\|f\|_{\infty}\|\eta\|_{\infty}.
\end{equation*}
Therefore, the function $fh_{u}\eta$ is bounded and measurable. According to \cite[Theorem 2.1]{CV16}, we know that for any $x\in E$ and any bounded measurable function
$g$ on $E$,
\begin{equation}
\label{2.33}
\lim_{t\to\infty}\mathbb{E}_{x}(g(X_{t})|T>t)=\int_{E}g(x)\nu(dx).
\end{equation}
So, by (\ref{2.1}) and (\ref{2.33}), we obtain
\begin{eqnarray*}
\liminf_{t\to\infty}\mathbb{E}_{x}(f(X_{qt})|T>t)
&\geq&\lim_{t\to\infty}\frac{e^{\lambda qt}\mathbb{E}_{x}[f(X_{qt})h_{u}(X_{qt})\eta(X_{qt}){\bf1}_{\{T>qt\}}]}{e^{\lambda t}\mathbb{P}_{x}(T>t)}\\
&=&\int_{E}f(x)h_{u}(x)\eta(x)\nu(dx).
\end{eqnarray*}
Based on the monotone convergence theorem, by letting $u\rightarrow\infty$ in the above formula, we have
\begin{equation}
\label{2.2}
\liminf_{t\to\infty}\mathbb{E}_{x}(f(X_{qt})|T>t)\geq\int_{E}f(x)m(dx).
\end{equation}
On the other hand, since $f$ is bounded, we can repeat the argument, replacing $f$ by $\|f\|_{\infty}-f$, which gives
\begin{equation}
\label{2.3}
\limsup_{t\to\infty}\mathbb{E}_{x}(f(X_{qt})|T>t)\leq\int_{E}f(x)m(dx).
\end{equation}
Combining (\ref{2.2}) and (\ref{2.3}), for positive and bounded function $f$, we have
\begin{equation}
\label{2.4}
\lim_{t\to\infty}\mathbb{E}_{x}(f(X_{qt})|T>t)=\int_{E}f(x)m(dx).
\end{equation}
For (\ref{2.4}), we can extend it to arbitrary bounded $f$ by subtraction.
\par
Finally, by change of variable in the Lebesgue integral and the dominated convergence theorem, we get
\begin{eqnarray*}
\lim_{t\rightarrow\infty}\mathbb{E}_{x}\left(\frac{1}{t}\int_{0}^{t}f(X_{s})ds|T>t\right)
&=&\lim_{t\rightarrow\infty}\mathbb{E}_{x}\left(\int_{0}^{1}f(X_{qt})dq|T>t\right)\\
&=&\lim_{t\rightarrow\infty}\int_{0}^{1}\mathbb{E}_{x}(f(X_{qt})|T>t)dq\\
&=&\int_{E}f(x)m(dx).
\end{eqnarray*}
Thus, we have proved that there exists a quasi-ergodic distribution for the process $X$.
\par
If Assumption (A) holds, then we know from \cite[Theorem 3.1]{CV16} that the $Q$-process admits the unique invariant probability measure
\begin{equation*}
\beta(dx)=\eta(x)\nu(dx).
\end{equation*}
Thus, $m$ coincides with the unique stationary distribution $\beta$ of the $Q$-process. This ends the proof of the theorem.
\end{proof}
\begin{remark}
From Theorem $\ref{thm2.1}$, one can easily see that $\nu\lesssim m$ $($see definition below$)$ as soon as $\eta$ is increasing.
\end{remark}
According to \cite[Theorem 2.1]{CV16}, we know that Assumption (A) implies that for all probability measure $\mu$ on $E$ and all $A\in\mathcal {E}$,
\begin{equation*}
\lim_{t\rightarrow\infty}\mathbb{P}_{\mu}(X_t\in A|T>t)=\nu(A).
\end{equation*}
From \cite[Proposition 1.2]{CV16}, we also know that Assumption (A) implies that for all probability measure $\mu$ on $E$,
\begin{equation*}
\lim_{t\to\infty}e^{\lambda t}\mathbb{P}_{\mu}(T>t)=\int_{E}\eta(x)\mu(dx).
\end{equation*}
Thus, by using a similar argument as in the proof of Theorem \ref{thm2.1}, we have the following result.
\begin{corollary}
Assume that Assumption $(\mathrm{A})$ is satisfied. Then, for any initial distribution $\mu$ on $E$ and any bounded measurable function $f$ on $E$, we have
\begin{equation*}
\lim_{t\rightarrow\infty}\mathbb{E}_{\mu}\left(\frac{1}{t}\int_{0}^{t}f(X_s)ds|T>t\right)=\int_{E}f(x)m(dx),
\end{equation*}
where $m$ is as in Theorem $\ref{thm2.1}$.
\end{corollary}

\section{Birth--death processes}
\label{sect3}
In this section, we study the quasi-ergodic distribution of birth--death processes. It is presented to illustrate that our main result is valid by using a new proof method which is different from the one used in the main result. Moreover, we also give a comparison between the quasi-ergodic distribution and the quasi-stationary distribution for the birth--death process.
\par
Let $X=(X_{t},t\geq0)$ be a continuous-time birth--death process taking values in $\mathbb{Z}_{+}:=\{0\}\cup\mathbb{N}$, where $0$ is an absorbing state and $\mathbb{N}=\{1,2,\cdots\}$ is an
irreducible transient class. Its jump rate matrix $Q:=(q_{ij},i,j\in\mathbb{Z}_{+})$ satisfies
\begin{equation*}
  q_{ij}=\left\{\begin{array}{ll} b_i&{\rm if}\ j=i+1,i\geq0,\\ d_i&{\rm if}\ j=i-1,i\geq1,\\  -(b_i+d_i)&{\rm if}\ j=i,i\geq0,\\ 0&{\rm otherwise}, \end{array}\right.
\end{equation*}
where the birth rates $(b_i,i\in\mathbb{N})$ and death rates $(d_i,i\in\mathbb{N})$ are strictly positive, and $d_{0}=b_{0}=0$.
\par
Define the potential coefficients $\pi=(\pi_i,i\in\mathbb{N})$ by
\begin{equation}
\label{3.1}
 \pi_1=1~~~~~~\mathrm{and}~~~~~~\pi_i=\frac{{{b}}_{1}{{b}}_{2}\cdots{{b}}_{i-1}}{{{d}}_{2}{{d}}_{3}\cdots{{d}}_{i}},~~~~i\geq2.
\end{equation}
Then, we have $b_i\pi_i=d_{i+1}\pi_{i+1}$, for $i\in\mathbb{N}$.
\par
Put
\begin{equation}
\label{3.2}
A=\sum_{i=1}^\infty\frac{1}{b_i\pi_i},~~~B=\sum_{i=1}^\infty{\pi}_{i},
~~~R=\sum_{i=1}^\infty\frac{1}{b_i\pi_i}\sum_{j=1}^{i}\pi_{j},~~~S=\sum_{i=1}^\infty\frac{1}{b_i\pi_i}\sum_{j=i+1}^{\infty}\pi_{j}.
\end{equation}
We observe that
\begin{equation*}
R+S=AB,~~~~A=\infty\Rightarrow R=\infty,~~~~S<\infty\Rightarrow B<\infty.
\end{equation*}
If absorption at 0 is certain which means that $\mathbb{P}_{i}(T<\infty)=1$, for $i\in\mathbb{N}$, where $T=\inf\{t\geq0: X_t=0\}$ is the absorption time of $X$, then it is equivalent to $A=\infty$ (see \cite{KM57a}). Therefore, $A=\infty$ implies the process $X$ is non-explosive. In this section, we assume that $A=\infty$. Note that, if absorption at 0 is certain, then $S<\infty$ is equivalent to Assumption (A) (see \cite[Theorem 4.1]{CV16}).
\par
We write $P_{ij}(t)=\mathbb{P}_{i}(X_{t}=j)$. It is well known (see, e.g., \cite[Theorem 5.1.9]{A91}) that under our assumptions, there exists a parameter $\lambda\geq0$, called
the decay parameter of the process $X$, such that
\begin{equation}
\label{3.3}
\lambda=-\lim_{t\to\infty}\frac{1}{t}\log P_{ij}(t),~~i,j\in\mathbb{N}.
\end{equation}
In \cite[Theorem 3.2]{VD91}, van Doorn proved that: (i) if $\infty$ is an entrance boundary (i.e., $R=\infty, S<\infty$), then $\lambda>0$ and there is a unique quasi-stationary distribution for the process $X$;
(ii) if $\infty$ is a natural boundary (i.e., $R=\infty, S=\infty$), then either $\lambda>0$ and there is an infinite continuum of quasi-stationary distributions, or $\lambda=0$ and there is no quasi-stationary distribution.
\par
Let $(Q_{i}(x),i\geq0)$ be the birth--death polynomials, given as
\begin{equation}
\label{3.4}
 \begin{array}{ll}
 &Q_{0}(x)=0,\\
 &Q_{1}(x)=1,\\
 &b_{i}Q_{i+1}(x)-(b_{i}+d_i)Q_{i}(x)+d_iQ_{i-1}(x)=-x Q_{i}(x),~~~~~i\in\mathbb{N}.
 \end{array}
\end{equation}
It is well known (see, e.g., \cite{C78}) that $Q_{i}(x)$ has $i-1$ positive, simple zeros, $x_{ij}~(j=1,2,\cdots,i-1)$, which verify the interlacing property
\begin{equation}
\label{3.5}
0<x_{i+1,j}<x_{i,j}<x_{i+1,j+1},~~j=1,2,\cdots,i-1,~i\geq2.
\end{equation}
Therefore, the following limits
\begin{equation}
\label{3.6}
\xi_{j}\equiv\lim_{i\to\infty}x_{ij},~~j\geq1,
\end{equation}
exist and satisfy $0\leq\xi_{j}\leq\xi_{j+1}<\infty$. It is easy to see from (\ref{3.4}) that, as a result,
\begin{equation}
\label{3.7}
x\leq\xi_{1}\Longleftrightarrow Q_{i}(x)>0~~~~\mathrm{for~all}~~i\in\mathbb{N}.
\end{equation}
 When $A=\infty$ and $S<\infty$, we have $0<\lambda=\xi_{1}<\xi_{2}<\cdots$ with $\lim_{j\to\infty}\xi_{j}=\infty$.
\par
Of importance to us in the proof of our main result will be the process $X$ conditioned to never be absorbed, usually referred to as the $Q$-process. Let $\overline{P}_{ij}(t)=\mathbb{P}_{i}(Y_{t}=j)$ be transition kernel of the $Q$-process $Y=(Y_{t},t\geq0)$. From \cite{HMS03}, we know that $\overline{P}_{ii}(t)=e^{\lambda t}P_{ii}(t)$, for all $i\in\mathbb{N}$. Thus, the $\lambda$-classification of the killed process $X^T$ can be presented in the following form. If $Y$ is positive recurrent (resp. recurrent, null recurrent, transient), then the killed process $X^T$ is said to be $\lambda$-positive (resp. $\lambda$-recurrent, $\lambda$-null, $\lambda$-transient).
\par
For two probability vectors $\rho=(\rho(i),i\in\mathbb{N})$ and $\rho^{\prime}=(\rho^{\prime}(i),i\in\mathbb{N})$,
we put $\rho^{\prime}\lesssim\rho$ and said that $\rho^{\prime}$ is stochastically smaller than $\rho$ in the sense of monotone likelihood-ratio ordering if and only if $(\rho(i)/\rho^{\prime}(i),i\in\mathbb{N})$ is increasing.
\par
For the birth--death process, we have the following result.
\par
\begin{theorem}
\label{thm3.1}
Let $X$ be a birth--death process for which $0$ is an absorbing state and $\infty$ is an entrance boundary. Then, there exists a quasi-ergodic distribution
$m=(m_{i}, i\in\mathbb{N})$ for the process $X$, where
\begin{equation*}
m_{i}=\frac{\pi_{i}Q_{i}^{2}(\lambda)}{\sum_{j\in\mathbb{N}}\pi_{j}Q_{j}^{2}(\lambda)}.
\end{equation*}
In particular, $m$ is the unique stationary distribution of the $Q$-process.
Moreover, $\nu\lesssim m$, where $\nu$ is the unique quasi-stationary distribution of the process $X$.
\end{theorem}
\begin{proof}
We first prove that the $Q$-process $Y$ is strongly ergodic, which implies the killed process $X^T$ is $\lambda$-positive. Although for general, almost surely absorbed Markov processes, Champagnat and Villemonais have proved that the process $Y$ is exponentially ergodic (see \cite[Theorem 3.1]{CV16}), we will prove that the process $Y$ is strongly ergodic by using a different proof method here. According to \cite[Proposition 5.9]{CMS13}, the process $Y$,
whose law starting from $i\in\mathbb{N}$ is given by
\begin{equation*}
\mathbb{P}_{i}(Y_{s_{1}}=i_{1},\cdots,Y_{s_{k}}=i_{k}):=\lim_{t\rightarrow\infty}\mathbb{P}_{i}(X_{s_{1}}=i_{1},\cdots,X_{s_{k}}=i_{k}|T>t),
\end{equation*}
is a Markov chain with transition kernel
\begin{equation}
\label{3.8}
\forall i,j\in \mathbb{N}:~~~~\mathbb{P}_{i}(Y_{s}=j)=e^{\lambda s}\frac{Q_{j}(\lambda)}{Q_{i}(\lambda)}\mathbb{P}_{i}(X_{s}=j).
\end{equation}
From (\ref{3.8}), we get that the process $Y$ is still a birth--death process taking values in $\mathbb{N}$, and its birth and death parameters are given respectively by
\begin{equation*}
\forall i\in \mathbb{N}:~~~~\overline{{b}}_{i}=\frac{Q_{i+1}(\lambda)}{Q_{i}(\lambda)}{b}_{i},
\end{equation*}
\begin{equation*}
\forall i\in \mathbb{N}:~~~~\overline{{d}}_{i}=\frac{Q_{i-1}(\lambda)}{Q_{i}(\lambda)}{d}_{i}.
\end{equation*}
So, we can compute the potential coefficients $\overline{\pi}=(\overline{\pi}_{i},i\in \mathbb{N})$ analogous to (\ref{3.1}): $\overline{\pi}_{1}=1$ and
\begin{equation*}
\overline{\pi}_{i}=\frac{\overline{{b}}_{1}\overline{{b}}_{2}\cdots\overline{{b}}_{i-1}}{\overline{{d}}_{2}\overline{{d}}_{3}\cdots\overline{{d}}_{i}}=Q_{i}^{2}(\lambda){\pi}_{i},~~~i\geq 2.
\end{equation*}
Similarly, we can compute the constants $\overline{A}, \overline{B}, \overline{R}, \overline{S}$ analogous to (\ref{3.2}):
\begin{align*}
\overline{A}&=\sum_{i=1}^\infty\frac{1}{Q_{i+1}(\lambda)Q_{i}(\lambda)b_i\pi_i},   &
\overline{B}&=\sum_{i=1}^{\infty}Q_{i}^{2}(\lambda){\pi}_{i},                      \\
\overline{R}&=\sum_{i=1}^{\infty}\frac{1}{Q_{i+1}(\lambda)Q_{i}(\lambda)b_i\pi_i}\sum_{j=1}^{i}Q_{j}^{2}(\lambda){\pi}_{j},&
\overline{S}&=\sum_{i=1}^{\infty}\frac{1}{Q_{i+1}(\lambda)Q_{i}(\lambda)b_i\pi_i}\sum_{j=i+1}^{\infty}Q_{j}^{2}(\lambda){\pi}_{j}.
\end{align*}
\par
From (\ref{3.4}), we have
\begin{equation}
\label{3.9}
b_{i}(Q_{i+1}(\lambda)-Q_{i}(\lambda))+d_i(Q_{i-1}(\lambda)-Q_{i}(\lambda))=-\lambda Q_{i}(\lambda).
\end{equation}
Multiplying both sides of (\ref{3.9}) by ${\pi}_{i}$, and then sum of $i$ from 1 to $k$, we get
\begin{equation*}
b_{k}{\pi}_{k}(Q_{k+1}(\lambda)-Q_{k}(\lambda))-d_1{\pi}_{1}Q_{1}(\lambda)=-\lambda\sum_{i=1}^{k}{\pi}_{i}Q_{i}(\lambda).
\end{equation*}
Note that $Q_{1}(\lambda)=1$, ${\pi}_{1}=1$ and $\lambda\sum_{i\in\mathbb{N}}\pi_{i}Q_{i}(\lambda)=d_1$ by (\ref{3.10}). Then
\begin{equation*}
Q_{k+1}(\lambda)-Q_{k}(\lambda)=\frac{\lambda}{b_{k}{\pi}_{k}}\sum_{i=k+1}^{\infty}{\pi}_{i}Q_{i}(\lambda)>0,~~~~~k\in\mathbb{N}.
\end{equation*}
Therefore, $Q_{i}(\lambda)$ is strictly increasing with $i\in\mathbb{N}$ and has the minimum 1. Also, we know from \cite[Lemma 3.4]{GM15}
that $Q_{i}(\lambda)$ is bounded, denoted by $W$ an upper bound. Thus, we have
\begin{eqnarray*}
\overline{R}&\geq&\sum_{i=1}^{\infty}\frac{1}{Q_{i+1}(\lambda)Q_{i}(\lambda)b_i\pi_i}\sum_{j=1}^{i}{\pi}_{j}\\
            &\geq&\sum_{i=1}^{\infty}\frac{1}{W^{2}b_i\pi_i}\sum_{j=1}^{i}{\pi}_{j}\\
            &=&\frac{1}{W^{2}}R\\
            &=&\infty
\end{eqnarray*}
and
\begin{eqnarray*}
\overline{S}&\leq&\sum_{i=1}^{\infty}\frac{1}{Q_{i+1}(\lambda)Q_{i}(\lambda)b_i\pi_i}\sum_{j=i+1}^{\infty}W^{2}{\pi}_{j}\\
            &\leq& W^{2}\sum_{i=1}^{\infty}\frac{1}{b_i\pi_i}\sum_{j=i+1}^{\infty}{\pi}_{j}\\
            &=&W^{2}S\\
            &<&\infty.
\end{eqnarray*}
Note that 1 is a reflecting boundary for the process $Y$. Hence, we know from \cite[Theorem 3.1]{M02} that the process $Y$ is strongly ergodic. Therefore, the killed process $X^T$ is $\lambda$-positive. And, it is well known that there exists a unique stationary distribution $(\frac{\pi_{i}Q_{i}^{2}(\lambda)}{\sum_{j\in\mathbb{N}}\pi_{j}Q_{j}^{2}(\lambda)}, i\in\mathbb{N})$ for the process $Y$.
From \cite{VD91}, we know that $Q(\lambda):=(Q_{i}(\lambda),i\in\mathbb{N})$ is a $\lambda$-invariant function for $Q$, that is, $QQ(\lambda)=-\lambda Q(\lambda)$, and the process $X$ admits $\nu=(\nu_{i},i\in\mathbb{N})$ as the unique quasi-stationary distribution, where
\begin{equation}
\label{3.10}
\nu_{i}=\frac{\pi_{i}Q_{i}(\lambda)}{\sum_{j\in\mathbb{N}}\pi_{j}Q_{j}(\lambda)}=\frac{\lambda}{{d}_{1}}\pi_{i}Q_{i}(\lambda).
\end{equation}
This implies that the series $\sum_{i\in\mathbb{N}}\pi_{i}Q_{i}(\lambda)$ is summable, and $\theta=(\theta_{i},i\in\mathbb{N})$ is the unique $\lambda$-invariant measure for $Q$, where $\theta_{i}=\pi_{i}Q_{i}(\lambda)$, that is, $\theta Q=-\lambda \theta$.
\par
Because the killed process $X^T$ is $\lambda$-positive, we know from \cite[Theorem 5.2.8]{A91} that
\begin{equation}
\label{6.1}
\lim_{t\rightarrow\infty}e^{\lambda t}P_{ij}(t)=\frac{Q_{i}(\lambda)\pi_{j}Q_{j}(\lambda)}{\sum_{k\in\mathbb{N}}\pi_{k}Q_{k}^{2}(\lambda)}.
\end{equation}
Also, we know from the proof of \cite[Proposition 5.9]{CMS13} that
\begin{equation}
\label{6.2}
\lim_{t\rightarrow\infty}\frac{\mathbb{P}_{j}(T>t)}{\mathbb{P}_{i}(T>t)}=\frac{Q_{j}(\lambda)}{Q_{i}(\lambda)}.
\end{equation}
Thus, for any $i\in\mathbb{N}$, we have
\begin{eqnarray*}
\lim_{t\rightarrow\infty}\mathbb{E}_{i}\left(\frac{1}{t}\int_{0}^{t}{\bf1}_{\{X_{s}=j\}}ds|T>t\right)
&=&\lim_{t\rightarrow\infty}\frac{1}{t}\int_{0}^{t}\frac{P_{ij}(s)\sum_{k\neq0}P_{jk}(t-s)}{\sum_{k\neq0}P_{ik}(t)}ds\\
&=&\lim_{t\rightarrow\infty}\frac{1}{t}\int_{0}^{t}\frac{e^{\lambda s}P_{ij}(s)\sum_{k\neq0}e^{\lambda(t-s)}P_{jk}(t-s)}{\sum_{k\neq0}e^{\lambda t}P_{ik}(t)}ds\\
&=&\lim_{x\rightarrow\infty}e^{\lambda x}P_{ij}(x)\lim_{y\rightarrow\infty}\frac{\sum_{k\neq0}e^{\lambda y}P_{jk}(y)}{\sum_{k\neq0}e^{\lambda y}P_{ik}(y)}\\
&=&\lim_{x\rightarrow\infty}e^{\lambda x}P_{ij}(x)\lim_{y\rightarrow\infty}\frac{\mathbb{P}_{j}(T>y)}{\mathbb{P}_{i}(T>y)}\\
&=&\frac{\pi_{j}Q_{j}^{2}(\lambda)}{\sum_{k\in\mathbb{N}}\pi_{k}Q_{k}^{2}(\lambda)}.
\end{eqnarray*}
Hence, we know that there exists a quasi-ergodic distribution $m$ for the process $X$.
\par
Notice that
\begin{eqnarray*}
\frac{m_{i}}{\nu_{i}}=\frac{\frac{\pi_{i}Q_{i}^{2}(\lambda)}{\sum_{j\in\mathbb{N}}\pi_{j}Q_{j}^{2}(\lambda)}}{\frac{\lambda}{{d}_{1}}\pi_{i}Q_{i}(\lambda)}
=\frac{Q_{i}(\lambda)}{\frac{\lambda}{{d}_{1}}\overline{B}}.
\end{eqnarray*}
Furthermore, from the proof of the above result, we know that $Q_{i}(\lambda)$ is increasing with $i\in\mathbb{N}$. Therefore, $({m_{i}}/{\nu_{i}},i\in\mathbb{N})$ is increasing. Thus, the result follows.
\end{proof}

\section{An example}
\label{sect5}
Set $d_{0}=b_{0}=0; b_i=i+3, d_i=(i+1)^{2}, i\in\mathbb{N}$. Then $\pi_1=1,$
\begin{align*}
\pi_i&=\frac{2(i+2)}{3(i+1)!},\\
 A&=\sum\limits_{i=1}^\infty\frac{1}{b_i\pi_i}=\frac{3}{2}\sum\limits_{i=1}^\infty\frac{(i+1)!}{(i+3)(i+2)}=\infty,\\
 S&=\sum\limits_{i=1}^\infty\frac{1}{b_i\pi_i}\sum\limits_{j=i+1}^{\infty}\pi_{j}=\sum\limits_{i=1}^\infty\frac{(i+1)!}{(i+3)(i+2)}\sum\limits_{j=i+1}^{\infty}\frac{(j+2)}{(j+1)!}<\infty.
\end{align*}
Moreover, the decay parameter $\lambda=2$ and the corresponding eigenfunction is $Q_{i}(\lambda)=\frac{i}{i+2}$ (see \cite[Example 5.1]{GM15}). From \cite[Theorem 3.2]{VD91}, we know that
there exists a unique quasi-stationary distribution $\nu=(\nu_{i}, i\in\mathbb{N})$ for the birth--death process $X$, where
\begin{equation*}
\nu_{i}=\frac{i}{(i+1)!}.
\end{equation*}
By Theorem \ref{thm3.1}, we know that $m=(m_{i}, i\in\mathbb{N})$ is the quasi-ergodic distribution of the process $X$, where
\begin{equation*}
m_{i}=\frac{\pi_{i}Q_{i}^{2}(\lambda)}{\sum_{j\in\mathbb{N}}\pi_{j}Q_{j}^{2}(\lambda)}=\frac{\frac{i^{2}}{(i+2)(i+1)!}}{\sum\limits_{j\in\mathbb{N}}\frac{j^{2}}{(j+2)(j+1)!}}.
\end{equation*}
Further, it is easy to check that
\begin{equation*}
\frac{m_{i}}{\nu_{i}}=\frac{\frac{i}{i+2}}{\sum\limits_{j\in\mathbb{N}}\frac{j^{2}}{(j+2)(j+1)!}}
\end{equation*}
is increasing with $i\in\mathbb{N}$, which implies $\nu\lesssim m$. Thus, this example illustrates our results.

\section*{Acknowledgements}
This work was supported by the National Natural Science Foundation of China (Grant No.
11371301) and the Key Laboratory of Hunan Province for New Retail Virtual Reality Technology
(Grant No. 2017TP1026).


\baselineskip 0.23in



\begin{thebibliography}{00}

\bibitem{A91}
{ Anderson, W. J.} (1991).
{\em Continuous-Time Markov Chains.}
Springer, New York.%

\bibitem{BR99}
{ Breyer, L. A. and Roberts, G. O.} (1999).
A quasi-ergodic theorem for evanescent processes.
{\em Stochastic Process. Appl.} {\bf 84,} 177--186.%

\bibitem{CV16}
{ Champagnat, N. and Villemonais, D.} (2016).
Exponential convergence to quasi-stationary distribution and $Q$-process.
{\em Probab. Theory Related Fields} {\bf 164,} 243--283.%

\bibitem{CD13}
{ Chen, J. and Deng, X.} (2013).
Large deviations and related problems for absorbing Markov chains.
{\em Stochastic Process. Appl.} {\bf 123,} 2398--2418.%

\bibitem{CLJ12}
{ Chen, J., Li, H. and Jian, S.} (2012).
Some limit theorems for absorbing Markov processes.
{\em J. Phys. A: Math. Theor.} {\bf 45,} 345003 (11pp).%

\bibitem{C78}
{ Chihara, T. S.} (1978).
{\em An Introduction to Orthogonal Polynomials.}
Gordon and Breach, New York.%

\bibitem{CMS13}
{ Collet, P., Mart\'{\i}nez, S. and San Mart\'{\i}n, J.} (2013).
{\em Quasi-Stationary Distributions.}
Probab. Appl. (N. Y.), Springer, Heidelberg.%

\bibitem{F74}
{ Flaspohler, D. C.} (1974).
Quasi-stationary distributions for absorbing continuous-time denumerable Markov chains.
{\em Ann. Inst. Statist. Math.} {\bf 26,} 351--356. %

\bibitem{GM15}
{ Gao, W.-J. and Mao, Y.-H.} (2015).
Quasi-stationary distribution for the birth-death process with exit boundary.
{\em J. Math. Anal. Appl.} {\bf 427,} 114--125.%

\bibitem{HMS03}
{ Hart, A., Mart\'{\i}nez, S. and San Mart\'{\i}n, J.} (2003).
The $\lambda$-classification of continuous-time birth-and-death process.
{\em Adv. in Appl. Probab.} {\bf 35,} 1111--1130.%

\bibitem{HZ16}
{ He, G. and Zhang, H.} (2016).
On quasi-ergodic distribution for one-dimensional diffusions.
{\em Statist. Probab. Lett.} {\bf 110,} 175--180.%

\bibitem{KM57a}
{ Karlin, S. and McGregor, J. L.} (1957a).
The classification of birth and death processes.
{\em Trans. Amer. Math. Soc.} {\bf 86,} 366--400.%

\bibitem{KM57b}
{ Karlin, S. and McGregor, J. L.} (1957b).
The differential equations of birth-and-death processes, and the Stieltjes moment problem.
{\em Trans. Amer. Math. Soc.} {\bf 85,} 489--546.%

\bibitem{M02}
{ Mao, Y.-H.} (2002).
Strong ergodicity for Markov processes by coupling methods.
{\em J. Appl. Probab.} {\bf 39,} 839--852.%

\bibitem{MV12}
{ M\'{e}l\'{e}ard, S. and Villemonais, D.} (2012).
Quasi-stationary distributions and population processes.
{\em Probab. Surv.} {\bf 9,} 340--410.%

\bibitem{VD91}
{ Van Doorn, E. A.} (1991).
Quasi-stationary distributions and convergence to quasi-stationarity of birth-death processes.
{\em Adv. in Appl. Probab.} {\bf 23,} 683--700.%

\bibitem{VDP13}
{ Van Doorn, E. A. and Pollett, P. K.} (2013).
Quasi-stationary distributions for discrete-state models.
{\em European J. Oper. Res.} {\bf 230,} 1--14.%

\bibitem{W85}
{ Whitt, W.} (1985).
The renewal-process stationary-excess operator.
{\em J. Appl. Probab.} {\bf 22,} 156--167.%

\bibitem{ZLS14}
{ Zhang, J., Li, S. and Song, R.} (2014).
Quasi-stationarity and quasi-ergodicity of general Markov processes.
{\em Sci. China Math.} {\bf 57,} 2013--2024.%



\end{thebibliography}
\end{document}